\documentclass[11pt]{article}

\usepackage{amsmath,amsthm,amssymb,url}
\usepackage{tikz}
\usepackage{tkz-graph}
\usepackage{authblk}
\usepackage{xcolor}
\usepackage{longtable}
\usepackage{xfrac} 
\usepackage{relsize}
\usepackage{nccmath}

\usepackage[top=1.25in, bottom=1.25in, left=1.25in, right=1.25in]{geometry}

\newcommand{\trd}[1]{\textcolor{red}{#1}}

\newcommand{\zed}{{\ensuremath{\mathbb{Z}}}} 

\newcommand{\V}{\ensuremath{\mathcal{V}}}
\newcommand{\W}{\ensuremath{\mathcal{W}}}

\newcommand*{\Bcell}[2]{
\draw[fill=blue!10!white] (#2,1-#1)--++(1,0)--++(0,1)--++(-1,0)--cycle;
}

\newtheorem{theorem}{Theorem}[section]

\newtheorem{lemma}[theorem]{Lemma}

\theoremstyle{definition}

\newtheorem{example}{Example}[section]
\newtheorem{definition}{Definition}[section]

\newtheorem{remark}[theorem]{Remark}
\title{Strong External Difference Families and Classification of $\alpha$-valuations}
\author[1]{Donald L.\ Kreher}
\affil[1]{Department of Mathematical Sciences\\
Michigan Technological University\
Houghton, MI 49931-1295, U.S.A.}
\author[2]{Maura B.\ Paterson}
\author[3]{Douglas R.\ Stinson\thanks{D.R.\ Stinson's research is supported by  NSERC discovery grant RGPIN-03882.}}
\affil[2]{School of Computing and Mathematical Sciences, Birkbeck, University of London, Malet St, London WC1E 7HX, UK}
\affil[3]{David R.\ Cheriton School of Computer Science\\University of Waterloo\\ Waterloo ON, N2L 3G1\\Canada}

\begin{document}
\maketitle

\begin{abstract}
One method of constructing $(a^2+1, 2,a, 1)$-SEDFs (i.e., strong external difference families) in $\mathbb{Z}_{a^2+1}$ makes use of $\alpha$-valuations of complete bipartite graphs 
$K_{a,a}$. We explore this approach and we provide a classification theorem which shows that all such $\alpha$-valuations can be constructed recursively via a sequence of ``blow-up'' operations. We also enumerate all $(a^2+1, 2,a, 1)$-SEDFs in $\mathbb{Z}_{a^2+1}$ for $a \leq 14$ and we show that all these SEDFs are equivalent to $\alpha$-valuations via affine transformations. Whether this holds for all $a > 14$ as well is an interesting open problem. 
We also study SEDFs in dihedral groups, where we show that two known constructions are equivalent.
\end{abstract}

\section{Introduction}
\label{intro.sec}

\begin{definition}[\cite{rosa}]
A {\em $\beta$-valuation} of a graph $G$ with $n$ edges is a one-to-one map $\V$ of the vertices into the set of integers $\{0,1,\dotsc,n\}$, such that, if we label each edge by the absolute value of the differences of the labels of the corresponding vertices, then the resulting edge labels are precisely the elements of the set $\{1,2,\dotsc,n\}$.  (This is also known as a {\em graceful labelling}.)
\end{definition}

\begin{definition}[\cite{rosa}]
An {\em $\alpha$-valuation} of a graph $G$ with $n$ edges is a $\beta$-valuation that satisfies the additional condition that there is some value $x$ with $0\leq x\leq n$ such that each edge is incident with one vertex whose label is at most $x$, and one whose label is greater than $x$.
\end{definition}
A graph with an $\alpha$-valuation $\V$ is necessarily bipartite, with the partition given by the set $V^{\sf large}$, which consists  of vertices whose labels are larger than $x$, and $V^{\sf small}$, which consists of vertices whose labels are at most $x$. Note that the value of $x$ is uniquely determined, as the edge whose vertices have difference $1$ must have labels $x$ and $x+1$.
For convenience, we will write $\V = (V^{\sf small}, V^{\sf large})$.

\medskip

It is often convenient to identify the vertices and their labels, and we will do this throughout the paper.

\begin{example}
\label{canonical.exam}
\cite{rosa}
We obtain an $\alpha$-valuation of the complete bipartite graph $K_{a,a}$ by defining $V^{\sf small} = \{0,1, \dots , a-1\}$ and
$V^{\sf large} = \{ a,2a, \dots , a^2\}$.
\end{example}

\begin{lemma}\label{lem:nointdiffrep}
Let $(V^{\sf small},V^{\sf large})$ be an $\alpha$-valuation of a complete bipartite graph.  Then there do not  exist distinct elements $x,y\in V^{\sf small}$ and distinct $u,v\in V^{\sf large}$ with $x-y=u-v$.
\end{lemma}
\begin{proof}
If we had $x,y\in V^{\sf small}$ with $x-y=d$ and $u,v\in V^{\sf large}$ with $u-v=d$ for $d\neq 0$, then we would have 
\begin{align*}
u-x&=(v+d)-(y+d),\\
&=v-y,
\end{align*}
which cannot occur in an $\alpha$-valuation.
\end{proof}

Let $a,b$ be positive integers.  We define $\Phi_{ab}\colon \{0,1,2,\dotsc,ab\}\rightarrow\{0,1,2,\dotsc,ab\}$ by setting $\Phi_{ab}(x)=ab-x$.  Then $\Phi_{ab}$ is bijective, and $\Phi^2_{ab}(x)=x$ for all $x\in\{0,1,2,\dotsc,ab\}$.
\begin{definition}
\label{equiv.def}
Let $\mathcal{V}_1$ and $\mathcal{V}_2$ be $\alpha$-valuations of the complete bipartite graph $K_{a,b}$.  We say that $\mathcal{V}_1$ and $\mathcal{V}_2$ are {\em equivalent} if $\mathcal{V}_1=\mathcal{V}_2$, or if for each vertex in $K_{a,b}$, the label in $\mathcal{V}_2$ is obtained by applying $\Phi_{ab}$ to its label in $\mathcal{V}_1$.
\end{definition}
We observe that when $\mathcal{V}_1$ and $\mathcal{V}_2$ are equivalent, $\Phi_{ab}$ maps the labels of $V^{\sf small}$ for $\mathcal{V}_1$ to the labels of $V^{\sf large}$ for $\mathcal{V}_2$, and vice versa. 

\subsection{Our contributions}

The rest of the paper is organized as follows.
In Section \ref{classify.sec}, we show that $\alpha$-valuations of $K_{a,b}$ have a rich structure and we prove a classification theorem for them. In particular, we show that all $\alpha$-valuations of $K_{a,b}$ can be constructed by a sequence of the  ``blowup'' operations that were described in \cite{PS24}.
In Section \ref{SEDFs.sec}, we turn our attention to strong external difference families (SEDFs) consisting of two sets of size $a$ in $\zed_{a^2+1}$. We enumerate all such SEDFs for $a \leq 14$ and show that they are all equaivlent to  $\alpha$-valuations of $K_{a,a}$. 
In Section \ref{dihedral.sec}, we review two constructions for SEDFs in dihedral groups and we show that they are equivalent.

\section{Classification of $\alpha$-valuations of $K_{a,b}$}
\label{classify.sec}

The following is our main classification theorem.
\begin{theorem}
\label{classify.thm}
Let $\mathcal{V}$ be an $\alpha$-valuation of $K_{a,b}$ with $ab>1$.  Then there exists a positive integer $\ell\geq 2$ such that $\mathcal{V}$ is equivalent to an $\alpha$-valuation with the following properties:
\begin{enumerate}
\item The elements of $V^{\sf large}$ are all multiples of $\ell$.
\item The set $V^{\sf small}$ is a union of ``runs'' of $\ell$ consecutive integers, where each run starts with a multiple of $\ell$.
\end{enumerate}
\end{theorem}

We prove Theorem \ref{classify.thm} by proving a sequence of lemmas.

\begin{lemma}
\label{lemab}
Let $\mathcal{V}$ be an $\alpha$-valuation of $K_{a,b}$ with $ab>1$. Then $ab \in V^{\sf large}$ and
$0 \in V^{\sf small}$.
\end{lemma}
\begin{proof}
The only way to obtain the difference $ab$ is as $ab-0$, where $ab \in V^{\sf large}$ and
$0 \in V^{\sf small}$.
\end{proof}

\begin{lemma}
\label{lem3}
Under the same hypotheses as Theorem {\rm \ref{classify.thm}}, there is an integer  $\ell \geq 2$ such that $\mathcal{V}$ is equivalent to an $\alpha$-valuation that satisfies the following properties, where $x$ is the largest element $x \in V^{\sf small}$.
\begin{enumerate}
\item $x+1, x+\ell+1 \in V^{\sf large}$, 
\item $x+2, \dots , x+ \ell \notin V^{\sf large}$,
\item $x-1, x-2, \dots , x - \ell+1 \in V^{\sf small}$, and
\item $x-\ell  \not\in V^{\sf small}$.
\end{enumerate}
\end{lemma}

\begin{proof}
Consider the difference $1$. It must occur as $1 = (x+1) - x$, where $x$ is the largest element of $V^{\sf small}$ and $x+1$ is the smallest element of $V^{\sf large}$. 

The difference $2$ can only occur in two possible ways:
\begin{description}
\item[case 1:] $2 = (x+2) - x$, or 
\item[case 2:] $2 = (x+1) - (x-1)$.
\end{description}
Only one of these possibilities actually occurs because every difference only occurs once. 

We  show that the two cases are equivalent. Assume the second case holds and suppose we apply  $\Phi_{ab}$. Then
$x \mapsto ab-x$ and $x+1 \mapsto ab-x-1$. In the transformed valuation, $V^{\sf small} \subseteq \{0, \dots ,y\}$ and $V^{\sf large} \subseteq \{y+1, \dots ,ab\}$, where $y = ab-x-1$. The transformation $\Phi_{ab}$ also switches case 1 and case 2.
If we are in case 1, the difference $2 = (x+2) - x$ is mapped to $2 = (y+1) - (y-1)$.
If we are in case 1, the difference $2 = (x+1) - (x-1)$ is mapped to $2 = (y+2) - y$. 

Therefore we assume without loss of generality (applying $\Phi_{ab}$ if necessary) that case 2 holds; hence
$x - 1 \in V^{\sf small}$ and $x + 2 \notin V^{\sf large}$.

Let $\ell$ be the smallest positive integer such that $x - \ell \not\in V^{\sf small}$. Then $x, x-1, \dots , x-\ell+1 \in V^{\sf small}$. Note that $\ell > 1$.

We have already proven 3.\ and 4. To prove 2., we note that pairs of  elements $x, x-1, \dots , x-\ell+1 \in V^{\sf small}$ yield differences $1, \dots , \ell-1$. Then, from Lemma~\ref{lem:nointdiffrep}, $x+2, \dots , x+ \ell \notin V^{\sf large}$ because $x+1 \in V^{\sf large}$.

To complete the proof, we note that the only way that the difference $\ell+1$ can possibly occur is 
$\ell + 1 = (x + \ell+1) - x$, so $x + \ell +1\in V^{\sf large}$.
\end{proof}

So far, we have the following partial structure, where red text is used to denote elements  that are \emph{not} in the relevant 
sets $V^{\sf small}$ or $V^{\sf large}$: 
\[
\begin{array}{|c|c|}
\hline
\rule[-.3\baselineskip]{0pt}{.18in} V^{\sf small} & V^{\sf large}\\ \hline
\trd{x - \ell} & x+\ell+1 \\ \hline
x - \ell + 1 & \trd{x + \ell} \\ \hline
x - \ell + 2 & \trd{x + \ell-1} \\ \hline
\vdots & \vdots \\ \hline
x-1 & \trd{x + 2} \\ \hline
x & x+1\\ \hline
\end{array}
\]
Note that there is a run of $\ell$ consecutive elements in $V^{\sf small}$ and a difference $\ell$ between consecutive elements in 
$V^{\sf large}$. This pattern continues, as we show in the next lemmas.

\begin{lemma}
\label{lem4}
Under the same hypotheses as Theorem {\rm \ref{classify.thm}}, $\mathcal{V}$ is equivalent to an $\alpha$-valuation in which no run of elements in $V^{\sf small}$ has length greater than  $\ell$.
\end{lemma}
\begin{proof} Suppose $a, a-1, \dots , a-k \in V^{\sf small}$ with $k \geq \ell$. Then, in particular,  $a, a - \ell \in V^{\sf small}$. We also have $x+1, x+\ell+1 \in V^{\sf large}$. Because $a - (a - \ell) = \ell = x+\ell+1 - (x+1)$, Lemma \ref{lem:nointdiffrep} is violated. We conclude that all runs of elements of  $V^{\sf small}$ have length at most  $\ell$.
\end{proof}

\begin{lemma}\label{lem5}
Under the same hypotheses as Theorem {\rm \ref{classify.thm}}, $\mathcal{V}$ is equivalent to an $\alpha$-valuation in which every element of $V^{\sf small}$ occurs as part of a run of exactly $\ell$ consecutive elements.
\end{lemma}

\begin{proof}
Lemma \ref{lem4} establishes that no run in  $V^{\sf small}$ can have length greater than $\ell$.
Suppose that there is a run in  $V^{\sf small}$ having length less than $\ell$. 
Let $a$ be the largest element of $V^{\sf small}$ that does not occur in a run of $\ell$  consecutive elements. Let $b$ be the largest element less than $a$ that does not occur in the run containing $a$. Then $b+1, \dots ,a$ is a (maximal) run of length $a-b < \ell$.  

Consider the difference $(x + 1) - b$. This difference must occur in the $\alpha$-valuation as a 
difference $y-z$ with $y \in V^{\sf large}$ and $z \in V^{\sf small}$.
Also, we have
\[ z = y - (x+1) + b > y - (x+1) + a - \ell.\]
It is impossible that $y = x+1$ because then $z = b$  and $b \not\in V^{\sf small}$.
Hence, $y \geq x + \ell + 1$ and then $z>a$. 

If $z+(a-b) \in V^{\sf small}$ then \[y-(z+(a-b))=(y-z)-a+b=(x+1)-a.\] 
However, as $x+1 \in V^{\sf large}$, $a \in V^{\sf small}$ and $y \neq x+1$, this implies the  $\alpha$-valuation  has a repeat of the difference $x + 1 - a$. Hence, $z+(a-b) \not\in V^{\sf small}$.

We have $z > a$, so $z$ is contained in a run of $\ell$ consecutive elements of $V^{\sf small}$, say the run
$d, d+1, \dots ,d+\ell - 1$. Hence $a+2 \leq d \leq z \leq d + \ell - 1$. Also,
$z+(a-b) \geq d + \ell$ because $z+(a-b) \not\in V^{\sf small}$.
This implies $z+(a-b)-\ell \geq d$. Because $z+(a-b)-\ell < z$, 
we must have 
$z+(a-b)-\ell \in V^{\sf small}$. 

We have
\[y-(z+(a-b)-\ell)=x+1-b-a+b+\ell=x+\ell+1-a.\] As $x+\ell+1\in V^{\sf large}$, this too is a repeated difference in the $\alpha$-valuation unless $y = x+\ell+1$ and $z+(a-b)-\ell = a$. 

Assume $z+(a-b)-\ell = a$. Then $z + a-b = a + \ell$. We have $z+(a-b) \not\in V^{\sf small}$, so $a + \ell = z + a-b \geq d + \ell$.
This implies $a \geq d$, which is clearly impossible because $a \leq d-2$.
\end{proof}

We now consider the differences between consecutive runs of $V^{\sf small}$.
Suppose the smallest element of one run in  $V^{\sf small}$ is $s$ and the smallest element in the next run is $s+d$. We refer to this as a \emph{gap of length $d$}. Note that a gap of length $d$ means that there are $d-\ell$ consecutive elements that are \emph{not} in $V^{\sf small}$, namely, $s+\ell, \dots , s+d-1$.
 
\begin{lemma}\label{lem6}
Under the same hypotheses as Theorem {\rm \ref{classify.thm}}, $\mathcal{V}$ is equivalent to an $\alpha$-valuation in which the gaps between consecutive runs of elements of $V^{\sf small}$ all have lengths that are  multiples of $\ell$. 
\end{lemma}
\begin{proof}
Suppose there is a gap of length $d$ between two consecutive runs in $V^{\sf small}$, where $d \not\equiv 0 \bmod \ell$. 
Let these runs consist of elements \[s, \dots , s+\ell-1 \quad \text{and} \quad t, \dots , t+\ell-1,\] where $t = s+d$. The
 $d-\ell$ consecutive elements  $s+\ell, \dots , s+d-1$ are  {not} in $V^{\sf small}$.

Consider the $d-\ell$ consecutive differences in the set 
\[\mathcal{D} = \{x+1 - (s+\ell), \dots , x+1 - (t-1)\}.\]
For every difference $y -z$ with $y \in V^{\sf large}$ and $z \in V^{\sf small}$, we obtain $\ell$ consecutive differences because $z$ is contained in a run of size $\ell$. It is clearly impossible to cover all the differences in $\mathcal{D}$ with disjoint runs of $\ell$ consecutive differences, because $d \not\equiv 0 \bmod \ell$. 
%
\end{proof}

Because all the runs of elements of $V^{\sf small}$ have length $\ell$ and  $|V^{\sf small}| = a$, it follows that $a \equiv 0 \bmod \ell$.

\begin{lemma} \label{lem7}
Under the same hypotheses as Theorem {\rm \ref{classify.thm}}, $\mathcal{V}$ is equivalent to an $\alpha$-valuation where $x$ is the largest element $x \in V^{\sf small}$ satisfies the congruence $x+1 \equiv 0 \bmod \ell$. Further, the largest element in any run in  $V^{\sf small}$ is congruent to $-1$ modulo $\ell$.
\end{lemma}

\begin{proof}
All the runs in $V^{\sf small}$ have length $\ell$ and all the gaps have lengths that are a multiple of $\ell$. The first run starts at $0$, from Lemma \ref{lemab}.
Hence  every run of elements of $V^{\sf small}$ has the form $d, d+1, d+\ell-1$ where $d \equiv 0 \bmod \ell$. 
Because the last run is $x - \ell + 1, \dots, x$, it follows that $x - \ell + 1 \equiv 0 \bmod \ell$, or $x+1 \equiv 0 \bmod \ell$. \end{proof}

\begin{lemma}\label{lem8}
Under the same hypotheses as Theorem {\rm \ref{classify.thm}}, $\mathcal{V}$ is equivalent to an $\alpha$-valuation in which the elements of $V^{\sf large}$ are all multiples of $\ell$.
\end{lemma}
\begin{proof}
We have shown that $x+1$, the smallest element in $V^{\sf large}$, is divisible by $\ell$. Suppose that $V^{\sf large}$ contains at least one element that is not divisible by $\ell$, and let $w$ be the smallest such element. We can write $w = x+1 + i \ell + d$, where
$0 < d < \ell$. Note that $w = x+1 + i \ell \not\in V^{\sf large}$, because then $d$ would be an internal difference in $V^{\sf large}$; this is not possible from Lemma \ref{lem:nointdiffrep} because $1, \dots , \ell-1$ are all internal differences in $V^{\sf small}$.

Consider the difference $i\ell+1$. This must occur as a difference $y - z$ with $y \in V^{\sf large}$ and $z \in V^{\sf small}$.
Note that $x+1 + i \ell - x = i\ell + 1$. We cannot have $y = x+1 + i \ell$, because $x+1 + i \ell \not\in V^{\sf large}$,
Hence $z < x$ and $y < x+1 + i \ell$, because $x$ is the largest element in $V^{\sf small}$. By assumption, $y$ must be a multiple of $\ell$, so $z \equiv -1 \bmod \ell$. It follows from the Lemmas \ref{lem5} and \ref{lem7} that $z - d \in V^{\sf small}$. We have
$y - (z-d) = d + i\ell + 1  = w - x$. Because $y < x+1 + i \ell < w$, we have two occurrences of the difference $d + i\ell + 1$, which is impossible. We conclude that all elements of $V^{\sf large}$ are  multiples of $\ell$.
\end{proof}

Now we can complete the proof of our main theorem. 

\begin{proof}[Proof of Theorem \ref{classify.thm}] 
This is an immediate consequence of Lemmas \ref{lemab}--\ref{lem8}.
\end{proof}

\medskip

If we apply $\Phi_{ab}$ to an $\alpha$-valuation of $K_{a,b}$ having the properties stated in Theorem \ref{classify.thm}, we obtain an equivalent $\alpha$-valuation that satisfies the following properties.

\begin{theorem}
\label{classify2.thm}
Let $\mathcal{V}$ be an $\alpha$-valuation of $K_{a,b}$ with $ab>1$.  Then there exists a positive integer $\ell>1$ such that $\mathcal{V}$ is equivalent to an $\alpha$-valuation with the following properties:
\begin{enumerate}
\item The elements of $V^{\sf small}$ are all multiples of $\ell$.
\item The set $V^{\sf large}$ is a union of ``runs'' of $\ell$ consecutive integers that \emph{end} with a multiple of $\ell$.
\end{enumerate}
\end{theorem}

We will call an $\alpha$-valuation that satisfies the properties listed in Theorem \ref{classify.thm} a \emph{type I} $\alpha$-valuation, while one that  satisfies the properties listed in Theorem \ref{classify2.thm} will be termed a \emph{type II} $\alpha$-valuation.

\begin{example}
Suppose $a = b = 3$, $V^{\sf small} = \{0,1,2\}$ and $V^{\sf large} = \{3,6,9\}$. This is a type I $\alpha$-valuation of $K_{3,3}$ with $\ell = 3$. It is equivalent to the type II $\alpha$-valuation consisting of $V^{\sf small} = \{0,3,6\}$ and $V^{\sf large} = \{7,8,9\}$. 
\end{example}

We describe two ``projection'' operations.

\begin{theorem}[{\sc Projection I}]
\label{proj1.thm}
Suppose $V^{\sf small}$ and $V^{\sf large}$ form a type I $\alpha$-valuation $\mathcal{V}$ of $K_{a,b}$. Replace every element $y \in V^{\sf large}$ by $y /\ell$ and call the resulting set $W^{\sf large}$.
Also, replace every run $R = i \ell, i \ell+1, \dots, i \ell + \ell - 1$ in $V^{\sf small}$ by the single element $i$ and call the resulting set $W^{\sf small}$. Then $W^{\sf small}$ and $W^{\sf large}$ form an $\alpha$-valuation $\mathcal{W}$ of $K_{a/\ell,b}$.
\end{theorem} 
\begin{proof}
It is clear that $W^{\sf small} \subseteq \{0, \dots , m/ \ell - 1 \}$ and $W^{\sf large} \subseteq \{m/ \ell, \dots , ab/ \ell \}$.
Further, $|W^{\sf small}| = |V^{\sf small}|/\ell = a/\ell$ and $|W^{\sf large}| = |V^{\sf large}| = b$. Now, consider  the differences in the valuation $\V$ obtained from an element $y \in V^{\sf large}$ and the run $R$:
\[ y - i \ell, y - (i \ell+1), \dots , y - (i \ell + \ell - 1).\] In the valuation $\W$, we get the single difference $y/\ell - i$.
We know that
\[ \{ y - r : y \in V^{\sf large}, r \in V^{\sf small}\} = \{1, \dots , ab\}.\]
It is therefore easy to see that
\[ \{ y' - z : y' \in W^{\sf large}, z \in W^{\sf small} \} = \{1, \dots , ab/\ell\}.\]
Finally, every element of $W^{\sf large}$ is greater than every element of $W^{\sf small}$.
Hence, $W^{\sf small}$ and $W^{\sf large}$ form an $\alpha$-valuation $\W$ of $K_{a/\ell,b}$.
\end{proof}

The following result is proven in a similar manner. We leave the details for the reader to verify.

\begin{theorem}[{\sc Projection II}]
\label{proj2.thm}
Suppose $V^{\sf small}$ and $V^{\sf large}$ form a type II  $\alpha$-valuation $\mathcal{V}$ of $K_{a,b}$. Replace every element $y \in V^{\sf small}$ by $y /\ell$ and call the resulting set $W^{\sf small}$.
Then, replace every run $i \ell - \ell + 1, i \ell - \ell + 2, \dots, i \ell$ in $V^{\sf large}$ by the single element $i$ and call the resulting set $W^{\sf large}$. Then $W^{\sf small}$ and $W^{\sf large}$ form an $\alpha$-valuation $\mathcal{W}$ of $K_{a,b/\ell}$.
\end{theorem}

After having computed $W^{\sf small}$ and $W^{\sf large}$ using a projection operation, we have a ``smaller'' $\alpha$-valuation. We can then project this smaller $\alpha$-valuation. The process continues until we reach a trivial $\alpha$-valuation of $K_{1,1}$.

\begin{example}
Again, suppose $a = b = 3$, $V^{\sf small} = \{0,1,2\}$ and $V^{\sf large} = \{3,6,9\}$. This is a type I $\alpha$-valuation of $K_{3,3}$ with $\ell = 3$. If we project this $\alpha$-valuation, we obtain the $\alpha$-valuation consisting of
$W^{\sf small} = \{0\}$ and $W^{\sf large} = \{1,2,3\}$. This is a type II $\alpha$-valuation of $K_{1,3}$ with $\ell = 3$. If we then project this  $\alpha$-valuation, we obtain the $\alpha$-valuation  of $K_{1,1}$ consisting of $X^{\sf small} = \{0\}$ and 
$X^{\sf large} = \{1\}$.
\end{example}

\subsection{Blowup Operations}

The lexicographic product $G \boldsymbol{\cdot} K_{\ell}^c$
replaces every vertex of a graph $G$ by $\ell$ independent vertices and it replaces every edge $xy$ by $\ell^2$ edges joining the $\ell$ copies of $x$ to the $\ell$ copies of $y$.
In \cite{PS24} two ``blowing-up'' operations were described, which were used to prove the following theorem.

\begin{theorem}
\label{lex.thm}
Suppose a graph $G$ has an $\alpha$-valuation, and let $\ell \geq 2$. 
Then $G \boldsymbol{\cdot} K_{\ell}^c$ has an $\alpha$-valuation.
\end{theorem}

When $G$ is a complete bipartite graph, the blowing-up operations described in \cite{PS24} constitute inverses to the two projection operations described above. We describe these now. 

\begin{theorem}[{\sc Blowup I}]
Suppose $W^{\sf small}$ and $W^{\sf large}$ form an $\alpha$-valuation of $K_{a,b}$. Let $\ell> 1$ be an integer.
Then there is an $\alpha$-valuation of $K_{\ell a,b}$.
\end{theorem}
\begin{proof}
First, multiply every element in $W^{\sf small} \cup W^{\sf large}$ by $\ell$.  Then, for every element  $\ell i \in V^{\sf small}$, replace it by the $\ell$ elements  $\{\ell i,\ell i+1,\dotsc,\ell i+(\ell-1)\}$.  We obtain an $\alpha$-valuation of $K_{\ell a,b}$.
\end{proof}

\begin{theorem}[{\sc Blowup II}]
Suppose $W^{\sf small}$ and $W^{\sf large}$ form an $\alpha$-valuation of $K_{a,b}$. Let $\ell> 1$ be an integer.
Then there is an $\alpha$-valuation of $K_{a,\ell b}$.
\end{theorem}
\begin{proof}
First, multiply every element in $W^{\sf small} \cup W^{\sf large}$ by $\ell$.  Then, for every element  $\ell i \in V^{\sf large}$, replace it by the $\ell$ elements  $\{\ell i,\ell i-1,\dotsc,\ell i-(\ell-1)\}$.  We obtain an $\alpha$-valuation of $K_{a,\ell b}$.
\end{proof}

If we are applying a sequence of Blowup operations, can assume without loss of generality that we will alternate Blowup operations of types I and II. This is because consecutive Blowup operations of same type can be replaced by a single Blowup operation. For example, {\sc Blowup I} with $\ell_1$, followed by {\sc Blowup I} with $\ell_2$, is identical to {\sc Blowup I} with $\ell_1\ell_2$.

It is easy to see that if we start with an $\alpha$-valuation of $K_{a,b}$ and we perform {\sc Blowup I} ({\sc Blowup II}, respectively) followed by {\sc Projection I} ({\sc Projection II}, respectively) using the same value of $\ell$, then we recover the initial $\alpha$-valuation.
Similarly, if we start with an $\alpha$-valuation of $K_{a,b}$ and we perform {\sc Projection I} ({\sc Projection II}, respectively)
followed by {\sc Blowup I} ({\sc Blowup II}, respectively) using the same value of $\ell$, then we recover the initial $\alpha$-valuation.

The above discussion immediately yields the following classification theorem.

\begin{theorem}
Every $\alpha$-valuation of $K_{a,b}$ can be obtained from the trivial $\alpha$-valuation of $K_{1,1}$ by an alternating sequence of {\sc Blowup I} and {\sc Blowup II} operations.
\end{theorem}

\begin{example}
\label{44.exam}
Suppose we start with the trivial $\alpha$-valuation of $K_{1,1}$ consisting of $\{0\}$ and $\{1\}$.
Now apply two blowup operations with $\ell = 4$, in the order {\sc Blowup II}, {\sc Blowup I}.
We obtain the following:
\[ 
\begin{array}{|c|c|c|}
\hline
\rule[-.3\baselineskip]{0pt}{.18in}\text{operation} & V^{\sf small} & V^{\sf large} \\ \hline
 & \{0\} & \{ 1 \} \\ \hline
 \text{\sc Blowup II} & \{0\} & \{ 4 \} \\
\ell = 4 & \{0\} & \{ 1,2,3,4 \} \\ \hline
  \text{\sc Blowup I} & \{0\} & \{ 4,8,12,16 \} \\
\ell = 4 & \{0,1,2,3\} & \{ 4,8,12,16 \} \\ \hline
  \end{array}
 \]
 The result is the $\alpha$-valuation obtained from Example \ref{canonical.exam}.
\end{example}

\begin{example}
\label{2222.exam}
Suppose we start with the trivial $\alpha$-valuation of $K_{1,1}$ consisting of $\{0\}$ and $\{1\}$.
Now apply four blowup operations, all with $\ell = 2$, in the order {\sc Blowup II}, {\sc Blowup I}, {\sc Blowup II}, {\sc Blowup I}.
We obtain the following:
\[ 
\begin{array}{|c|c|c|}
\hline
\rule[-.3\baselineskip]{0pt}{.18in}\text{operation} & V^{\sf small} & V^{\sf large} \\ \hline
 & \{0\} & \{ 1 \} \\ \hline
 \text{\sc Blowup II} & \{0\} & \{ 2 \} \\
\ell = 2 & \{0\} & \{ 1,2 \} \\ \hline
  \text{\sc Blowup I} & \{0\} & \{ 2,4 \} \\
\ell = 2 & \{0,1\} & \{ 2,4 \} \\ \hline
  \text{\sc Blowup II} & \{0,2\} & \{ 4,8 \} \\
\ell = 2 & \{0,2\} & \{ 3,4,7,8 \} \\ \hline
  \text{\sc Blowup I} & \{0,4\} & \{ 6,8,14,16 \} \\
\ell = 2 & \{0,1,4,5\} & \{ 6,8,14,16\} \\ \hline
 \end{array}
 \]
\end{example}

\begin{example}
\label{242.exam}
Suppose we start with the trivial $\alpha$-valuation of $K_{1,1}$ consisting of $\{0\}$ and $\{1\}$.
Now apply three blowup operations, in the order {\sc Blowup II}  with $\ell = 2$, {\sc Blowup I} with $\ell =4$, and {\sc Blowup II} with $\ell = 2$.
We obtain the following:
\[ 
\begin{array}{|c|c|c|}
\hline
\rule[-.3\baselineskip]{0pt}{.18in}\text{operation} & V^{\sf small} & V^{\sf large} \\ \hline
 & \{0\} & \{ 1 \} \\ \hline
 \text{\sc Blowup II} & \{0\} & \{ 2 \} \\
\ell = 2 & \{0\} & \{ 1,2 \} \\ \hline
  \text{\sc Blowup I} & \{0\} & \{ 4,8 \} \\
\ell = 4 & \{0,1,2,3\} & \{ 4,8 \} \\ \hline
  \text{\sc Blowup II} & \{0,2\} & \{ 4,8 \} \\
\ell = 2 & \{0,2,4,6\} & \{ 7,8,15,16 \} \\ \hline
 \end{array}
 \]
\end{example}

\section{Strong external difference families and $\alpha$-valuations of $K_{a,a}$}
\label{SEDFs.sec}

In this section,  we consider the connections between certain strong external difference families and $\alpha$-valuations of $K_{a,a}$. We begin with some relevant definitions.

For two disjoint subsets $A,B$ of an additive  group $G$, define the multiset $\mathcal{D}(B,A)$ as follows:
\[ \mathcal{D}(B,A) = \{ y - x: y \in B, x \in A \}.  \]

\begin{definition}[\cite{PS16}]
Let $G$ be an additive  group of order $n$. Suppose $m \geq 2$. 
An {\em $(n, m, \ell; \lambda)$-strong external difference family} (or {\em $(n, m, \ell; \lambda)$-SEDF}) is a set of $m$ disjoint $\ell$-subsets of $G$, say $\mathcal{A} = (A_0,\dots,A_{m-1})$, such that  the following multiset equation holds for each $i\in\{0,1,\dotsc,m-1\}$:
\[
\bigcup_{j\neq i} \mathcal{D}(A_i, A_j) = \lambda (G \setminus \{0\}).
\]
\end{definition}

Since they were first defined in 2016, there have been numerous papers that have studied SEDFs. Some references include 
\cite{BJWZ,HuPa,JeLi,LeLiPr,LePr,MaSt,WYFF}. Here, we will focus mainly on $(n, 2, \ell; 1)$-SEDFs in cyclic groups $\zed_n$. For SEDFs with these parameters, we must have $n = \ell^2+1$.

The proof of the following lemma is immediate.

\begin{lemma} 
An $\alpha$-valuation of $K_{a,a}$ is an $(a^2+1, 2, a; 1)$-SEDF in the cyclic group $\zed_{a^2+1}$.
\end{lemma}

\begin{remark}
{\rm
Analogously, an $\alpha$-valuation of $K_{a,b}$ with $a\neq b$ gives rise to a so-called \emph{generalised SEDF} (or GSEDF) in the cyclic group $\zed_{ab+1}$. This GSEDF has two sets, one of size $a$ and one of size $b$. In the rest of this paper, we will focus on the SEDF case (i.e., $a=b$).
}
\end{remark}

We now consider the notion of equivalence of SEDFs.

\begin{definition}[\cite{PS16}]
\label{equiv.defn}
Suppose that $(A_0,A_{1})$ and $(B_0,B_{1})$ are $(n, 2, \ell; 1)$-SEDFs in  $\zed_n$.
We say that these two SEDFs are \emph{affine equivalent} (or, more briefly, \emph{equivalent}) if there exist $\alpha \in {\zed_n}^*$ and $\beta \in \zed_n$ such that
\begin{enumerate}
\item $B_i = \alpha A_i + \beta$, for $i = 1,2$, or 
\item $B_i = \alpha A_{1-i} + \beta$, for $i = 1,2$. 
\end{enumerate}
The second part of this definition simply allows us to switch the roles of $A_0$ and $A_1$.
\end{definition}

Note that equivalence of $\alpha$-valuations, as defined in Definition \ref{equiv.def}, is a special case of (affine) equivalence of SEDFs.

\begin{example}
Consider the  
$\alpha$-valuation from Example \ref{242.exam} consisting of  $\{0,2,4,6\}$ and $\{ 7,8,15,16 \}$.
Under the mapping $x \mapsto 9x \bmod 17$, we obtain $\{0,1,2,3\}$ and $\{ 4,8,12,16 \}$, which is the 
$\alpha$-valuation from Example \ref{44.exam}.
\end{example}

\begin{example}
In \cite{HucJefNep}, Huczynska, Jefferson and  Nep\v{s}insk\'{a} establish that there are precisely two inequivalent 
$(17, 2, 4; 1)$-SEDFs in  $\zed_{17}$, namely, 
\begin{center}$\{0,1,2,3\}$\quad \text{and} \quad$\{4,8,12,16\}$;\end{center} 
and 
\begin{center}$\{1,4,13,16\}$ \quad \text{and} \quad $\{2,8,9,15\}$.\end{center}  
The first SEDF is the $\alpha$-valuation from Example \ref{44.exam} (see also Example \ref{canonical.exam}).
The second SEDF is not an $\alpha$-valuation, but it is equivalent to an $\alpha$-valuation.
The affine function $x \mapsto 6x+11 \bmod 17$ transforms the second SEDF to $\{0, 1, 4, 5\}$ and $\{6, 8, 14, 16\}$, 
which is the 
$\alpha$-valuation from Example \ref{2222.exam}.
\end{example}

\begin{example}
It is interesting to observe that different blowup sequences can yield equivalent SEDFs.
The blowup sequence $(3,2,2,3)$ generates the 
SEDF consisting of 
\[  \{0,1,2,6,7,8\} \quad \text{and} \quad \{9, 12, 21, 24, 33, 36\}. \]
The SEDF generated from blowup sequence $(2,3,3,2)$ is
\[  \{0, 1, 6, 7, 12, 13 \} \quad \text{and} \quad \{14, 16, 18, 32, 34, 36\}. \]
If we apply the mapping $x \mapsto 6x+2 \bmod 37$ to the second SEDF, we obtain the first SEDF.

In a similar manner, it can be shown that the two blowup sequences $(2,2,3,3)$ and $(3,3,2,2)$ generate equivalent SEDFs, as do the three blowup sequences $(2,6,3)$, $(3,6,2)$ and $(6,6)$.
\end{example}

\subsection{Enumeration of SEDFs}


Our enumeration depends on some interesting structural results that were proven in \cite{CGHK}. The paper \cite{CGHK} studies 
\emph{near-factorizations} of finite groups. A near-factorization of an additive group $G$ is a pair of (disjoint) subsets of $G$, say $(S,T)$, such that $S + T = G \setminus \{0\}$. It is clear that $(A,B)$ is an $(a^2+1, 2, a; 1)$-SEDF in  $\zed_{a^2+1}$ if and only if 
$(A, -B)$ is a near-factorization of  $\zed_{a^2+1}$ with $|A| = |B| = a$. Similarly, if $(A,B)$ is a GSEDF in $\zed_{ab+1}$  with $|A| = a$ and $|B| = b$, then $(A, -B)$ is a near-factorization of  $\zed_{ab+1}$ with $|A| =a$ and $|B| = b$, and conversely.

We use the following two results which follow immediately from \cite{CGHK}.

\begin{lemma}
\textup{\cite[Corollary 1]{CGHK}}
\label{lem1}
Suppose $(A,B)$ is an $(a^2+1, 2, a; 1)$-SEDF in  $\zed_{a^2+1}$. Then $-A = g + A$ and $-B = h + B$ for some $g, h \in \zed_{a^2+1}$.
\end{lemma}

From Lemma \ref{lem1}, the following can be proven.

\begin{lemma}
\textup{\cite[Proposition 2]{CGHK}}
\label{lem2}
Suppose $(A,B)$ is an $(a^2+1, 2, a; 1)$-SEDF in  $\zed_{a^2+1}$. Then there exists an element $g \in \zed_{a^2+1}$ such that
$g + A = -(g+A)$ and $g + B = -(g+B)$.
\end{lemma}

Lemma \ref{lem2} is very useful. It tells us that we can restrict our attention to SEDFs $(A,B)$ where $A = -A$ and $B = -B$, 
because any SEDF can be transformed into an SEDF of this form by a suitable translation. Such an SEDF will be termed \emph{symmetric}.

\begin{example}
The two sets $A = \{0,1,2,3\}$ \text{and} $B = \{4,8,12,16\}$ form a $(17, 2, 4; 1)$-SEDF in  $\zed_{17}$.
If we compute $A' = A + 7 = \{ 7,8,9,10\}$ and $B' = B + 7 = \{ 2,6,11,15\}$, we obtain an SEDF
$(A',B')$ in which  $A' = -A'$ and $B' = -B'$. Hence, $(A',B')$ is symmetric.
\end{example}

For $x \in \zed_{a^2+1}$, define $P_x = \{x, -x\}$. We note that $|P_x| = 1$ if $x = 0$, or if 
$a$ is odd and $x = (a^2+1)/2$, and $|P_x| = 2$, otherwise. It is clear that, in a  symmetric $(a^2+1, 2, a; 1)$-SEDF, say $(A,B)$, both $A$ and $B$ are a union of sets $P_x$. 

\begin{lemma} \mbox{\quad} \vspace{-.2in} \\
\label{pairs.lem}
\begin{enumerate}
\item Suppose $a$ is odd (so $a^2+1$ is even). 
Then $P_0 \subseteq A$ and $P_{(a^2+1)/2} \subseteq  B$, or $P_0 \subseteq B$ and $P_{(a^2+1)/2} \subseteq A$.
All other $P_x$'s contained in $A$ or $B$ have cardinality two.
\item Suppose that $a$ is even (so $a^2+1$ is odd).
Then  $A$ and $B$ are both unions of $P_x$'s of cardinality two.
\end{enumerate}
\end{lemma}
\begin{proof}
When $a$ is odd, both $A$ and $B$ must contain a $P_x$ where $|P_x| = 1$. There are two $P_x$'s with $|P_x| = 1$ when $a$ is odd. Because $A$ and $B$ are disjoint, it must be the case that one of $P_0$ and $P_{(a^2+1)/2}$ is contained in $A$ and the other  is contained in $B$. 

When $a$ is even, both $A$ and $B$ must contain an even number of $P_x$'s where $|P_x| = 1$. But there is only one $P_x$ with $|P_x| = 1$ when $a$ is even, namely, $P_0$.  It follows that neither $A$ nor $B$ contains $P_0$.
\end{proof}




\subsection{Enumeration algorithm}

We now describe the algorithm we used to enumerate the SEDFs that are presented in Table \ref{SEDFs.tab}. An $a$-subset $A\subseteq X=\zed_{a^2+1}$ is said to have 
a \emph{mate}, say $B$, if $B$ is an $a$-subset of $X$
such that  $\mathcal{D}(B,A) = X\setminus\{0\}$.
Of course $B$ must be disjoint from $A$.
Furthermore, by Lemma~\ref{pairs.lem}, we may assume that $A=-A$ and $B=-B$ (i.e., $A$ and $B$ are symmetric), 
and that both $A$ and $B$ are unions of sets $P_x$.

\begin{lemma}
\mbox{\quad}\vspace{-.2in}\\
\begin{enumerate} 
\item Suppose $x \neq y$ and $|P_x| = |P_y| = 2$. 
Then 
\begin{equation}
\label{P.eq}
\mathcal{D}(P_y,P_x) = \{ y - x, -y-x, y+x, -y + x\}  = P_{y-x} \cup P_{y+x}.
\end{equation}
\item Suppose $x \neq y$ and at least one of $P_x$ and $P_y$ has cardinality one.
Then 
\begin{equation}
\label{P2.eq}
\mathcal{D}(P_y,P_x) = P_{y-x} = P_{-y-x} = P_{y+x} = P_{-y+x}.
\end{equation}
\end{enumerate}
\end{lemma}

\begin{proof}
Part 1.\ is obvious, so we just discuss part 2. Suppose $|P_y| = 1$. Then $2y = 0$. We have 
$\mathcal{D}(P_y,P_x) = \{y-x,y+x\}.$ However, $(y-x) + (y+x) = 2y = 0$, so
$y + x = - (y-x)$. It follows that $\mathcal{D}(P_y,P_x) = P_{y-x} = P_{-y-x} = P_{y+x} = P_{-y+x}$.
The proof is similar when  $|P_x| = 1$.
\end{proof}

Given a fixed symmetric $A$, we want to choose  
$P_y$'s to place in potential mate $B$ so $\mathcal{D}(B,A)$ contains all the distinct $P_d$'s ($d \neq 0$) once each. Let $\mathcal{P}(A)=\{P_x: P_x \subseteq A\}$.

When $a$ is even, let $\mathcal{P}= \{P_x: 1 \leq x \leq a^2 / 2 \}$, so $|\mathcal{P}| = a^2 / 2$.
Here all sets $P_x$ contained in $A \cup B$ have cardinality two.

For odd values of $a$, let $\mathcal{P}= \{P_x: 1 \leq x \leq (a^2 +1) / 2 \}$, 
so $|\mathcal{P}| = (a^2 +1) / 2$.
Note that $|P_{(a^2 +1) / 2}| = 1$. 
We can assume without loss of generality that $P_0 \subseteq A$ (i.e., $0 \in A$)
and $P_{(a^2+1)/2} \subseteq B$ (i.e., $(a^2+1)/2 \in B$). 
This accounts for the difference $(a^2+1)/2$ in $\mathcal{D}(B,A)$.
All other sets $P_x$ contained in $A \cup B$ have cardinality two.


Suppose we regard $A$ as fixed and we consider a possible $P_y \in \mathcal{P}\setminus \mathcal{P}(A)$ to be placed in $B$. 
For a given $P_y \in \mathcal{P}\setminus \mathcal{P}(A)$ and a given $P_d \in \mathcal{P}$, we will denote by $M_A[P_d,P_y]$ the number of 
$P_x$'s in $A$ such that $P_d \subseteq \mathcal{D}(P_y,P_x)$. This defines a $|\mathcal{P}|$ by $|\mathcal{P}\setminus \mathcal{P}(A)|$ matrix $M_A$. This matrix is basically a Kramer-Mesner matrix (see \cite{KM76}) where the rows and columns are indexed by orbits under the action of the multiplicative group $\{1,-1\}$.

\begin{lemma}
\label{mat.lem}
Suppose $A$ is a  subset of $\mathcal{P}$. Suppose $P_y \in \mathcal{P}\setminus \mathcal{P}(A)$.
Then \begin{equation}
\label{M.eq}
 M_A[P_d,P_y]= |\{P_x \subseteq A : P_x = P_{d \pm y} \} |.
\end{equation}
\end{lemma}

\begin{proof}
Suppose that $P_x \subseteq A$ and $P_x = P_{d \pm y}$. 
Replacing $x$ by $-x$ if necessary, we can assume $x = d \pm y$. If $x = d-y$, then $d = x+y$ and $P_d \subseteq \mathcal{D}(P_y,P_x)$ from  (\ref{P.eq}) or (\ref{P2.eq}). Similarly,
if $x = d+y$, then $d = x-y$ and $P_d \subseteq \mathcal{D}(P_y,P_x)$.

Conversely, if $P_d \subseteq \mathcal{D}(P_y,P_x)$, then $d \in \{ y - x, -y-x, y+x, -y + x\}$. We consider each possible case in turn:
\begin{itemize}
\item 
If $d = y-x$, then $x = y-d$ and $P_{x} = P_{-x} = P_{d - y}$. 
\item If $d = -y-x$, then $x = -y-d$ and $P_{x}  = P_{-x} = P_{d+y}$. 
\item If $d = y+x$, then $x = d-y$ and $P_{x} =  P_{d - y}$. 
\item Finally, if $d = -y+x$, then $x = d+y$ and $P_{x} =  P_{d +y}$. 
\end{itemize}
In each case, $P_x = P_{d \pm y}$.
\end{proof}
We observe that, for given values of $y$ and $d$, there are at most two possible values of $x$ such that 
(\ref{M.eq}) is satisfied.


Let $J$ be the all-$1$'s vector of length $|\mathcal{P}|$.
In the following theorem $U = (u[P_y])$  will be a vector of length $|\mathcal{P}\setminus \mathcal{P}(A)|$ indexed by
\[\{ P_y \in \mathcal{P}\setminus \mathcal{P}(A)\} .\]

\begin{theorem}
\label{even-equiv.thm}
The matrix equation \begin{equation}
\label{matrix.eq}
M_AU^T=J^T
\end{equation}
 has a $(0,1)$-valued solution $U$
if and only
\[ B=\medmath{\bigcup} \: \{P_y \in \mathcal{P}\setminus \mathcal{P}(A) : u[P_y]=1 \}\] 
is a mate to $A$.
\end{theorem}

\begin{proof}
Suppose $U$ is a $(0,1)$-valued solution to the matrix equation (\ref{matrix.eq}). 
Define
\[ B=\medmath{\bigcup} \: \{P_y \in \mathcal{P}\setminus \mathcal{P}(A) : u[P_y]=1 \}.\] 
We first observe that $B \cap A = \emptyset$.
Let $d \neq 0$. Because (\ref{matrix.eq}) is satisfied, there must be a unique $y$ such that $M_A[P_d,P_y]u[P_y] = 1$.
Hence,  $M_A[P_d,P_y] = u[P_y] = 1$.
Because $u[P_y] = 1$, it follows from Lemma \ref{mat.lem} that there is a unique $P_x \in A$ such that $P_d \subseteq \mathcal{D}(P_y,P_x)$.
Hence $B$ is a mate to $A$.

Conversely, suppose $B$ is a mate to $A$ and define 
$U$ by $u[P_y]=1$ if and only if
$P_y \subseteq B$. 
Consider any $d \neq 0$. Because $(A,B)$ is an SEDF, there is a unique
$P_x\in A$ and $P_y \in B$ such that $P_d \subseteq \mathcal{D}(P_y,P_x)$.
Hence there is a unique $y$ such that $M_A[P_d,P_y]u[P_y] = 1$.
It follows easily that (\ref{matrix.eq}) holds.
\end{proof}

\begin{remark}
The matrix equation (\ref{matrix.eq}) has a $(0,1)$-valued solution if and only if there is a subset of the columns of $M_A$ whose sum is $J$. 
\end{remark}

\begin{remark}
If a column of $M_A$ contains an entry equal to 2, then the corresponding entry of $U$ must equal $0$ in any  $(0,1)$-valued solution to the matrix equation (\ref{matrix.eq}).
\end{remark}

\begin{example} Suppose $a=4$, so $a^2 + 1 = 17$, and suppose $A = \{P_1,P_3\}$.
The matrix $M_A$ has rows indexed by $P_1, \dots, P_8$ and columns indexed by $P_2, P_4, P_5,P_6,P_7,P_8$.

Consider the column labelled $P_4$. Because $A = \{P_1,P_3\}$, we are interested in values $d$ such that $P_1 = P_{d \pm 4}$ or 
 $P_3 = P_{d \pm 4}$. In the first case, $d = 3$ or $5$; in the second case, $d = 1$ or $7$.
Hence, we have a $1$ in rows $P_1$, $P_3$, $P_5$ and $P_7$ of column $P_4$.

It can be checked that the column labelled $P_5$ has $1$'s in rows $P_2$, $P_4$, $P_6$ and $P_8$.
The sum of columns  $P_4$ and $P_5$ of $M_A$ gives the all-$1$'s vector of length eight, so $B = \{P_4,P_5\}$ is a mate to $A$. This solution corresponds to the vector $U = (0,1,1,0,0,0)$.

The complete matrix $M_A$ is as follows:
\[ 
\begin{array}{c||c|c|c|c|c|c}
    & P_2 & P_4 & P_5 & P_6 & P_7 & P_8 \\ \hline\hline
P_1 & 2 &  1 & 0 & 0 & 0 & 0 \\\hline
P_2 & 0 &  0 & 1 & 0 & 0 & 0 \\\hline
P_3 & 1 &  1 & 0 & 1 & 0 & 0 \\\hline
P_4 & 0 &  0 & 1 & 0 & 1 & 0 \\\hline
P_5 & 1 &  1 & 0 & 1 & 0 & 1 \\\hline
P_6 & 0 &  0 & 1 & 0 & 1 & 1 \\\hline
P_7 & 0 &  1 & 0 & 1 & 1 & 1 \\\hline
P_8 & 0 &  0 & 1 & 1 & 1 & 1 
\end{array}
\]

Note that the column labelled $P_2$ has an entry equal to $2$. To compute all the values in column $P_2$,
we find the values $d$ such that $P_1 = P_{d \pm 2}$ or 
 $P_3 = P_{d \pm 2}$. In the first case, $d = 1$ or $3$; in the second case, $d = 1$ or $5$.
Hence, we have a $2$ in row $P_1$ and $1$'s in rows $P_3$ and $P_5$.
\end{example}

\begin{example}
Suppose $a=3$, so $a^2 + 1 = 10$, and suppose $A = \{P_0,P_1\}$.
The matrix $M_A$ has rows indexed by $P_1, \dots, P_5$ and columns indexed by $P_2, P_3, P_4, P_5$.
The complete matrix $M_A$ is as follows:
\[ 
\begin{array}{c||c|c|c|c}
    & P_2 & P_3 & P_4 & P_5 \\ \hline\hline
P_1 &   1 & 0 & 0 & 0  \\\hline
P_2 &   1 & 1 & 0 & 0  \\\hline
P_3 &   1 & 1 & 1 & 0  \\\hline
P_4 &   0 & 1 & 1 & 1  \\\hline
P_5 &   0 & 0 & 1 & 1  \\
\end{array}
\]
The sum of columns $P_2$ and $P_5$ of $M_A$ gives the all-$1$'s vector of length five, so $B = \{P_2,P_5\}$ is a mate to $A$. This solution corresponds to the vector $U = (0,0,0,1,1)$.
Of course we already knew that it must be the case that $P_5 \subseteq B$.

It might be of interest to verity how the entries in column $P_5$ are computed, because $|P_5| = 1$.
We find the values $d$ such that $P_0 = P_{d \pm 5}$ or 
 $P_1 = P_{d \pm 5}$. In the first case, $d = 5$, so we have a $1$ in row $P_5$. In the second case, $d = 4$ or $6$. Of course $P_4 = P_6$ and we only consider values of $d \leq 10/2 = 5$, so we have a $1$ entry in row $P_4$.
\end{example}

Solutions to  the matrix equation $M_AU^T = J^T$  can be found using a variety
of methods. We prefer the very fast 
backtracking algorithm using the ``dancing links'' data structure 
described in
\cite{Knuth} that was implemented by 
Petteri Kaski and  Olli Pottonen.
It can be downloaded from the web page \cite{libexact}.
We note that, before running this algorithm, it is necessary to delete any columns 
that contain an entry equal to 2.

A pair $(A,B)$ of $a$-element subsets of $\zed_{a^2+1}$
is a $(a^2+1,2,a;1)$-SEDF if $B$ is a mate to $A$.
Solutions to the matrix equation will yield many equivalent $(a^2+1,2,a;1)$-SEDFs, where equivalence is defined as in Definition \ref{equiv.defn}.

Let $v=a^2+1$, let $v_0 = \lfloor \frac{v-1}{2} \rfloor$ and let $a = \lfloor \frac{a}{2} \rfloor$.
Because we may assume $(A,B)$ is symmetric, we can assume that
\begin{equation}\label{A}
   A= \left\{\begin{array}{ll}
T \cup (v-T) & \text{if $a$ is even}\\
\{0\} \cup T \cup (v-T) & \text{if $a$ is odd},\\
\end{array}\right.
\end{equation}
where $T$ is an $a_0$-element subset of 
$\{1,2,\ldots,v_0\}$ and
\[
v-T = \{v-x: x \in T\}.
\]
We also note that if $A$ and $B$ are symmetric,
then $mA$ and $mB$ are  also symmetric for any non-zero unit
$m$ in $\zed_v^*$.  

\begin{enumerate}
\item 
For each $a_0$-subset $T \subseteq \{1,2,\ldots,v_0\}$,
we construct $A$ according to  (\ref{A}).
If there is a unit $m \in \zed_v^*$ such that (in lexicographical order)
\[mA \text{ \raisebox{-.03in}{$\stackrel{<}{\scriptsize{\textsc{lex}}}$} } A, \]
then we discard $A$ and proceed to the next $a_0$-subset $T$ in $\{1,2,\ldots,v_0\}$; otherwise we process $A$.

\item For each  $A$ that is to be processed,
we then compute $M_A$ and find (0,1) valued solutions to the matrix equation $M_AU^T=J^T$.  
Each such solution $U$ corresponds to a mate $B$ for $A$,
yielding a symmetric SEDF, $(A,B)$.
\item
Finally, we convert the  SEDF $(A,B)$ into canonical form
as follows.
\begin{enumerate}
\item We find 
\[
f_A(X) \in 
  \textsc{Aff}(v) = \{ x\mapsto \alpha x + \beta :
\alpha , \beta \in \zed_{v}\text{ and }\gcd(\alpha,v)=1 \}
\]
such that $f_A(A) \text{ \raisebox{-.03in}{$\stackrel{<}{\scriptsize{\textsc{lex}}}$} }
g(A)$
for all $g \in \textsc{Aff}(v)$.
\item We find 
\[
f_B(X) \in 
  \textsc{Aff}(v) = \{ x\mapsto \alpha x + \beta :
\alpha , \beta \in \zed_{v}\text{ and }\gcd(\alpha,v)=1 \}
\]
such that $f_B(B) \text{ \raisebox{-.03in}{$\stackrel{<}{\scriptsize{\textsc{lex}}}$} }
g(B)$
for all $g \in \textsc{Aff}(v)$.
\item If $f_A(A) \text{ \raisebox{-.03in}{$\stackrel{<}{\scriptsize{\textsc{lex}}}$} }
f_B(B)$,  then we report $(f_A(A),f_A(B))$ as the canonical form solution;
otherwise we report $(f_B(B),f_B(A))$.
\end{enumerate}
\item Different SEDFs may have the same canonical form. Thus, as a final step,
we remove the duplicate canonical form solutions.
\end{enumerate}


We have enumerated all inequivalent $(a^2+1, 2, a, 1)$-SEDFs in  $\zed_{a^2+1}$, for $1 \leq a \leq 14$.
These are presented in Table  \ref{SEDFs.tab}. For each SEDF, we present a symmetric representation and the canonical representation, along with the affine mapping that transforms the symmetric representation to the canonical representation.
The symmetric representation uses the notation $P_x$ defined above. If we apply the specified affine mapping
modulo $a^2+1$ to every element of the symmetric representation, we obtain the given canonical representation.


\begin{center}
\begin{longtable}{|c|c|c|}
\caption{Inequivalent $(a^2+1, 2, a, 1)$-SEDFs in  $\zed_{a^2+1}$}\\ \hline
\label{SEDFs.tab}
{number} & {solution} & {mapping} 
\\ \hline \hline 
\endfirsthead
\caption[]{(continued)}\\ \hline
{number} & {solution} & {mapping} 
\\ \hline \hline 
\endhead
\hline
3.1 &$\{P_{0},P_{1}\}$, $\{P_{2},P_{5}\}$ & X+1 \\ 
 & $\{0,1,2\}$,$\{3,6,9\}$ &  \\ 
\hline
\hline
4.1 &$\{P_{1},P_{3}\}$, $\{P_{4},P_{5}\}$ & 8X+10 \\ 
 & $\{0,1,2,3\}$,$\{4,8,12,16\}$ &  \\ 
\hline
4.2 &$\{P_{1},P_{4}\}$, $\{P_{2},P_{8}\}$ & 6X+11 \\ 
 & $\{0,1,4,5\}$,$\{6,8,14,16\}$ &  \\ 
\hline
\hline
5.1 &$\{P_{0},P_{1},P_{2}\}$, $\{P_{3},P_{8},P_{13}\}$ & X+2 \\ 
 & $\{0,1,2,3,4\}$,$\{5,10,15,20,25\}$ &  \\ 
\hline
\hline
6.1 &$\{P_{1},P_{3},P_{5}\}$, $\{P_{6},P_{7},P_{18}\}$ & 18X+21 \\ 
 & $\{0,1,2,3,4,5\}$,$\{6,12,18,24,30,36\}$ &  \\ 
\hline
6.2 &$\{P_{1},P_{2},P_{17}\}$, $\{P_{4},P_{10},P_{16}\}$ & 2X+4 \\ 
 & $\{0,1,2,6,7,8\}$,$\{9,12,21,24,33,36\}$ &  \\ 
\hline
\hline
7.1 &$\{P_{0},P_{1},P_{2},P_{3}\}$, $\{P_{4},P_{11},P_{18},P_{25}\}$ & X+3 \\ 
 & $\{0,1,2,3,4,5,6\}$,$\{7,14,21,28,35,42,49\}$ &  \\ 
\hline
\hline
8.1 &$\{P_{1},P_{3},P_{5},P_{7}\}$, $\{P_{8},P_{9},P_{24},P_{25}\}$ & 32X+36 \\ 
 & $\{0,1,2,3,4,5,6,7\}$,$\{8,16,24,32,40,48,56,64\}$ &  \\ 
\hline
8.2 &$\{P_{1},P_{6},P_{8},P_{15}\}$, $\{P_{5},P_{13},P_{23},P_{24}\}$ & 28X+38 \\ 
 & $\{0,1,2,3,8,9,10,11\}$,$\{12,16,28,32,44,48,60,64\}$ &  \\ 
\hline
8.3 &$\{P_{1},P_{4},P_{13},P_{16}\}$, $\{P_{2},P_{8},P_{26},P_{32}\}$ & 22X+43 \\ 
 & $\{0,1,4,5,16,17,20,21\}$,$\{22,24,30,32,54,56,62,64\}$ &  \\ 
\hline
\hline
9.1 &$\{P_{0},P_{1},P_{2},P_{3},P_{4}\}$, $\{P_{5},P_{14},P_{23},P_{32},P_{41}\}$ & X+4 \\ 
 & $\{0,1,2,3,4,5,6,7,8\}$,$\{9,18,27,36,45,54,63,72,81\}$ &  \\ 
\hline
9.2 &$\{P_{0},P_{1},P_{8},P_{9},P_{10}\}$, $\{P_{11},P_{14},P_{17},P_{38},P_{41}\}$ & X+10 \\ 
 & $\{0,1,2,9,10,11,18,19,20\}$,$\{21,24,27,48,51,54,75,78,81\}$ &  \\ 
\hline
10.1 &$\{P_{1},P_{3},P_{5},P_{7},P_{9}\}$, $\{P_{10},P_{11},P_{30},P_{31},P_{50}\}$ & 50X+55 \\ 
 & $\{0,1,2,3,4,5,6,7,8,9\}$,$\{10,20,30,40,50,60,70,80,90,100\}$ &  \\ 
\hline
10.2 &$\{P_{1},P_{2},P_{32},P_{35},P_{36}\}$, $\{P_{11},P_{16},P_{31},P_{38},P_{43}\}$ & 3X+7 \\ 
 & $\{0,1,2,3,4,10,11,12,13,14\}$,$\{15,20,35,40,55,60,75,80,95,100\}$ &  \\ 
\hline
\hline
11.1 &$\{P_{0},P_{1},P_{2},P_{3},P_{4},P_{5}\}$, $\{P_{6},P_{17},P_{28},P_{39},P_{50},P_{61}\}$ & X+5 \\ 
 & $\{0,1,2,3,4,5,6,7,8,9,10\}$,$\{11,22,33,44,55,66,77,88,99,110,121\}$ &  \\ 
\hline
\hline
12.1 &$\{P_{1},P_{3},P_{5},P_{7},P_{9},P_{11}\}$, $\{P_{12},P_{13},P_{36},P_{37},P_{60},P_{61}\}$ & 72X+78 \\ 
 & $\{0,1,2,3,4,5,6,7,8,9,10,11\}$ &  \\ 
 & $\{12,24,36,48,60,72,84,96,108,120,132,144\}$ &  \\ 
\hline
12.2 &$\{P_{1},P_{6},P_{27},P_{34},P_{55},P_{62}\}$, $\{P_{8},P_{22},P_{31},P_{45},P_{54},P_{68}\}$ & 57X+93 \\ 
 & $\{0,1,2,3,4,5,12,13,14,15,16,17\}$ &  \\ 
 & $\{18,24,42,48,66,72,90,96,114,120,138,144\}$ &  \\ 
\hline
12.3 &$\{P_{1},P_{2},P_{4},P_{5},P_{68},P_{71}\}$, $\{P_{7},P_{19},P_{31},P_{43},P_{55},P_{67}\}$ & 48X+98 \\ 
 & $\{0,1,2,3,4,5,18,19,20,21,22,23\}$ &  \\ 
 & $\{24,30,36,60,66,72,96,102,108,132,138,144\}$ &  \\ 
\hline
12.4 &$\{P_{1},P_{3},P_{13},P_{15},P_{17},P_{19}\}$, $\{P_{20},P_{21},P_{28},P_{29},P_{68},P_{69}\}$ & 72X+82 \\ 
 & $\{0,1,2,3,4,5,24,25,26,27,28,29\}$ &  \\ 
 & $\{30,36,42,48,78,84,90,96,126,132,138,144\}$ &  \\ 
\hline
12.5 &$\{P_{1},P_{3},P_{21},P_{23},P_{25},P_{27}\}$, $\{P_{28},P_{29},P_{36},P_{37},P_{44},P_{45}\}$ & 72X+86 \\ 
 & $\{0,1,2,3,12,13,14,15,24,25,26,27\}$ &  \\ 
 & $\{28,32,36,64,68,72,100,104,108,136,140,144\}$ &  \\ 
\hline
12.6 &$\{P_{1},P_{2},P_{17},P_{20},P_{35},P_{38}\}$, $\{P_{4},P_{10},P_{16},P_{58},P_{64},P_{70}\}$ & 8X+16 \\ 
 & $\{0,1,2,6,7,8,24,25,26,30,31,32\}$ &  \\ 
 & $\{33,36,45,48,81,84,93,96,129,132,141,144\}$ &  \\ 
\hline
12.7 &$\{P_{1},P_{10},P_{19},P_{35},P_{44},P_{53}\}$, $\{P_{12},P_{15},P_{25},P_{52},P_{56},P_{62}\}$ & 16X+22 \\ 
 & $\{0,1,2,6,7,8,36,37,38,42,43,44\}$ &  \\ 
 & $\{45,48,57,60,69,72,117,120,129,132,141,144\}$ &  \\ 
\hline
\pagebreak
\hline
13.1 &$\{P_{0},P_{1},P_{2},P_{3},P_{4},P_{5},P_{6}\}$, $\{P_{7},P_{20},P_{33},P_{46},P_{59},P_{72},P_{85}\}$ & X+6 \\ 
 & $\{0,1,2,3,4,5,6,7,8,9,10,11,12\}$ &  \\ 
 & $\{13,26,39,52,65,78,91,104,117,130,143,156,169\}$ &  \\ 
\hline
\hline
14.1 &$\{P_{1},P_{3},P_{5},P_{7},P_{9},P_{11},P_{13}\}$, $\{P_{14},P_{15},P_{42},P_{43},P_{70},P_{71},P_{98}\}$ & 98X+105 \\ 
 & $\{0,1,2,3,4,5,6,7,8,9,10,11,12,13\}$ &  \\ 
 & $\{14,28,42,56,70,84,98,112,126,140,154,168,182,196\}$ &  \\ 
\hline
14.2 &$\{P_{1},P_{2},P_{38},P_{41},P_{77},P_{78},P_{80}\}$, $\{P_{19},P_{26},P_{43},P_{64},P_{71},P_{81},P_{88}\}$ & 5X+10 \\ 
 & $\{0,1,2,3,4,5,6,14,15,16,17,18,19,20\}$ &  \\ 
 & $\{21,28,49,56,77,84,105,112,133,140,161,168,189,196\}$ &  \\ 
\hline
\end{longtable}
\end{center}

As an example to illustrate an affine mapping, 
consider solution 4.2 in Table  \ref{SEDFs.tab}. The given symmetric form is
\[A = \{P_{1},P_{4}\} = \{1,4,13,16\}, \quad B = \{P_{2},P_{8}\} = \{ 2,8,9,15\}.\]
The canonical form is  \[6A+11 = \{0,1,4,5 \}, \quad 6B+11 = \{6,8,14,16 \}.\]
It is obtained from the mapping $X \mapsto 6X+11$ in $\zed_{17}$.

All of the SEDFs listed in  Table  \ref{SEDFs.tab} turn out to be $\alpha$-valuations. 
In Table \ref{blowup.tab}, 
the column labelled ``sequence'' indicates the $\ell$-values used in a sequence of blowup 
operations that yield each SEDF. 
The solution numbers correspond to the ones in Table \ref{SEDFs.tab}. In every case, we begin with a 
{\sc Blowup II} operation and then alternate {\sc Blowup I} and {\sc Blowup II} operations. 
Note that the sequence of blowup operations for a given solution is not necessarily uniquely determined.

\begin{table}[thb]
\caption{SEDFs and Blowup Sequences}
\label{blowup.tab}
\begin{center}
\begin{tabular}{|c|c||c|c|}
\hline
\rule[-.3\baselineskip]{0pt}{.18in} {number} & {blowup sequence} & {number} & {blowup sequence} 
\\ \hline \hline 
3.1 &  $(3,3)$ &
4.1 &  $(4,4)$ \\ \hline
4.2 &  $(2,2,2,2)$ &
5.1 &  $(5,5)$ \\\hline
6.1 &  $(6,6)$ &
6.2 &  $(3,2,2,3)$ \\\hline
7.1 &  $(7,7)$ &
 8.1 &  $(8,8)$ \\\hline
8.2 &  $(2,2,4,4)$ &
8.3 &  $(2,2,2,2,2,2)$ \\\hline
 9.1 &  $(9,9)$&
9.2 &   $(3,3,3,3)$ \\\hline
10.1 &  $(10,10)$ &
10.2 &  $(5,2,2,5)$ \\\hline
11.1 &  $(11,11)$ &
12.1 & $(4,2,3,6)$\\\hline
12.2 & $(3,2,2,2,2,3) $ &
12.3 & $(3,2,4,6)$\\\hline
12.4 & $(12,12)$ &
12.5 & $(4,3,3,4)$  \\\hline
12.6 & $(6,2,2,6)$ &
12.7 & $(2,2,3,2,2,3)$ \\\hline
13.1 & $(13,13)$ 
& 14.1 & $(14,14)$ \\\hline
14.2 & $(7,2,2,7)$ & & \\\hline
\end{tabular}
\end{center}
\end{table}



\section{SEDFs in dihedral groups}
\label{dihedral.sec} 

We will be studying SEDFs and near factorizations in dihedral groups, so it is convenient to switch to multiplicative notation.
If $S$ and $T$ are subsets of a multiplicative finite group $G$ with identity $e$ such that
\[
ST = \{ st : s \in S, t \in T\} = G\setminus\{e\}, 
\]
then we say that $(S,T)$ is a \emph{near factorization} of $G$.
Near factorizations were, to our knowledge, first studied in de Caen {\it et al.}\ ~\cite{CGHK}. That paper was motivated by the
problem of factoring $J-I$ into a product of $(0,1)$ matrices.
Indeed, if $R$ is a subset of the group $G$, we may define the
matrix $M:G\times G \rightarrow \{0,1\}$ by $M[g,h] = 1$ if $gh^{-1} \in R$ and 0 otherwise. Then a near factorization $G-e=ST$ provides the $(0,1)$ matrix factorization
$M(S)M(T)=M(G\setminus\{e\})= J-I$, where $J$ is the matrix of all $1$'s. 

An $(n, 2, k, 1)$-\emph{strong external difference family} 
(or $(n, 2, k, 1)$-SEDF) in the multiplicative group $G$ is a pair 
of disjoint $k$-subsets $(A_1,A_2)$ of $G$, 
such that the multiset 
\[\{xy^{-1} : x \in A_1, y \in A_2\}\] contains each
element of $G-e$ exactly once. Thus $A_1 A_2^{-1} = G-e$
and hence $(A_1,A_2^{-1})$ is a near factorization.
Consequently, a $(n, 2, k, 1)$-SEDF in $G$ is equivalent to 
a near factorization of $G$ with both factors having size $k$.
Observe that
\begin{enumerate}
\item
if $ST=G-e$, then $T^{-1}S^{-1} = (ST)^{-1} = (G-e)^{-1}= G-e$; and
\item
if $ST=G-e$ and  $g,h \in G$, then 
$
(gSh)(h^{-1}Tg^{-1}) =
gSTg^{-1} = g(G-e)g^{-1} = G-e$.
\end{enumerate}
Hence we say that $(S,T)$, $(T^{-1},S^{-1})$ and 
$(gSh,h^{-1}Tg^{-1})$  are all \emph{equivalent} near factorizations of $G$.

\subsection{Example} 
\label{example}

We present an example based on \cite{CGHK}. Let $D_{13}$ be the dihedral group of order 26 with generators $a,b$ and relations
$a^{2}=1$, $b^{13}=1$ and $aba=b^{-1}$. 

\begin{enumerate}
\item
The elements of $D_{13}$ are
\[
\{a^ib^j:\text{$i=0,1$ and $j=0,1,2,\ldots,12$}\}.
\]
They  can be depicted by the following diagram:

\begin{center}
\tikz[scale=0.75, line width=1]
{
 \foreach \i in {0,1,2} { \draw (0,\i)--(13,\i); }
 \foreach \i in {0,...,13} {\draw (\i,0)--(\i,2);}
 \node at (-0.5,2.5) {\makebox(0,0){$i =$}};
 \node at (-0.5,1.5) {\makebox(0,0){$b^i$}};
 \node at (-0.5,0.5) {\makebox(0,0){$ab^i$}};
 \foreach \i in {0,...,12} 
 {
   \node at (\i+0.5,2.5) {$\i$};
 }
}
\end{center}

\item
Remove the identity and enter the sequence $1,2,3,4,5$ five times, starting in $a^1$ and wrapping around
the end, finishing at $b$. 

\begin{center}
\tikz[scale=0.75, line width=1]
{
 \foreach \i in {0,1,2} { \draw (0,\i)--(13,\i); }
 \foreach \i in {0,...,13} {\draw (\i,0)--(\i,2);}
 \node at (-0.5,2.5) {\makebox(0,0){$i =$}};
 \node at (-0.5,1.5) {\makebox(0,0){$b^i$}};
 \node at (-0.5,0.5) {\makebox(0,0){$ab^i$}};
 \foreach \i in {0,...,12} 
 {
   \node at (\i+0.5,2.5) {$\i$};
 }
 \draw[fill=gray] (0,1)--(1,1)--(1,2)--(0,2)--cycle;
 \node at ( 1.5, 1.5) {\makebox(0,0){1}};
 \node at ( 2.5, 1.5) {\makebox(0,0){2}};
 \node at ( 3.5, 1.5) {\makebox(0,0){3}};
 \node at ( 4.5, 1.5) {\makebox(0,0){4}};
 \node at ( 5.5, 1.5) {\makebox(0,0){5}};
 \node at ( 6.5, 1.5) {\makebox(0,0){1}};
 \node at ( 7.5, 1.5) {\makebox(0,0){2}};
 \node at ( 8.5, 1.5) {\makebox(0,0){3}};
 \node at ( 9.5, 1.5) {\makebox(0,0){4}};
 \node at (10.5, 1.5) {\makebox(0,0){5}};
 \node at (11.5, 1.5) {\makebox(0,0){1}};
 \node at (12.5, 1.5) {\makebox(0,0){2}};
 \node at (12.5, 0.5) {\makebox(0,0){3}};
 \node at (11.5, 0.5) {\makebox(0,0){4}};
 \node at (10.5, 0.5) {\makebox(0,0){5}};
 \node at ( 9.5, 0.5) {\makebox(0,0){1}};
 \node at ( 8.5, 0.5) {\makebox(0,0){2}};
 \node at ( 7.5, 0.5) {\makebox(0,0){3}};
 \node at ( 6.5, 0.5) {\makebox(0,0){4}};
 \node at ( 5.5, 0.5) {\makebox(0,0){5}};
 \node at ( 4.5, 0.5) {\makebox(0,0){1}};
 \node at ( 3.5, 0.5) {\makebox(0,0){2}};
 \node at ( 2.5, 0.5) {\makebox(0,0){3}};
 \node at ( 1.5, 0.5) {\makebox(0,0){4}};
 \node at ( 0.5, 0.5) {\makebox(0,0){5}};
}
\end{center}

\item
Partition the cells into tiles of the same shape  that each  contain exactly one
cell of each type. 

\begin{center}
\tikz[scale=0.75, line width=1]
{
 \node at (-0.5,2.5) {\makebox(0,0){$i =$}};
 \node at (-0.5,1.5) {\makebox(0,0){$b^i$}};
 \node at (-0.5,0.5) {\makebox(0,0){$ab^i$}};
 \foreach \i in {0,...,12} 
 {
   \node at (\i+0.5,2.5) {$\i$};
 }
 \draw[fill=red!20!white] (0,0)--(13,0)--(13,2)--(0,2)--cycle;
 \draw[fill=gray] (0,1)--(1,1)--(1,2)--(0,2)--cycle;
 \foreach \i in {0,5,10}
 {
  \draw[fill=blue!10!white](\i,0)--(3+\i,0)--(3+\i,2)--(1+\i,2)--(1+\i,1)--(\i,1)--cycle;
 }
 \node at ( 1.5, 1.5) {\makebox(0,0){1}};
 \node at ( 2.5, 1.5) {\makebox(0,0){2}};
 \node at ( 3.5, 1.5) {\makebox(0,0){3}};
 \node at ( 4.5, 1.5) {\makebox(0,0){4}};
 \node at ( 5.5, 1.5) {\makebox(0,0){5}};
 \node at ( 6.5, 1.5) {\makebox(0,0){1}};
 \node at ( 7.5, 1.5) {\makebox(0,0){2}};
 \node at ( 8.5, 1.5) {\makebox(0,0){3}};
 \node at ( 9.5, 1.5) {\makebox(0,0){4}};
 \node at (10.5, 1.5) {\makebox(0,0){5}};
 \node at (11.5, 1.5) {\makebox(0,0){1}};
 \node at (12.5, 1.5) {\makebox(0,0){2}};
 \node at (12.5, 0.5) {\makebox(0,0){3}};
 \node at (11.5, 0.5) {\makebox(0,0){4}};
 \node at (10.5, 0.5) {\makebox(0,0){5}};
 \node at ( 9.5, 0.5) {\makebox(0,0){1}};
 \node at ( 8.5, 0.5) {\makebox(0,0){2}};
 \node at ( 7.5, 0.5) {\makebox(0,0){3}};
 \node at ( 6.5, 0.5) {\makebox(0,0){4}};
 \node at ( 5.5, 0.5) {\makebox(0,0){5}};
 \node at ( 4.5, 0.5) {\makebox(0,0){1}};
 \node at ( 3.5, 0.5) {\makebox(0,0){2}};
 \node at ( 2.5, 0.5) {\makebox(0,0){3}};
 \node at ( 1.5, 0.5) {\makebox(0,0){4}};
 \node at ( 0.5, 0.5) {\makebox(0,0){5}};
 \foreach \i in {0,1,2} { \draw (0,\i)--(13,\i); }
 \foreach \i in {0,...,13} {\draw (\i,0)--(\i,2);}
}
\end{center}
\item
Let $A$ be the group elements 
in the leftmost tile:
\[
A=\{b,b^2,ab^2,ab,a\}.
\]
Each tile has a ``notch.'' Let $B$ be the group elements corresponding to these notches:
\[
B=\{e,ab^5,b^5,ab^{10},b^{10}\}.
\]
\end{enumerate}
Then $AB=D_{13}\setminus\{1\}$ and hence it is a near factorization: 
\[
\begin{array}{c|ccccc}
\cdot&e      &ab^{ 5}& b^{ 5}&ab^{10}& b^{10}\\\hline
 b     & b     &ab^{ 4}& b^{ 6}&ab^{ 9}& b^{11}\\
 b^{ 2}& b^{ 2}&ab^{ 3}& b^{ 7}&ab^{ 8}& b^{12}\\
ab^{ 2}&ab^{ 2}& b^{ 3}&ab^{ 7}& b^{ 8}&ab^{12}\\
ab     &ab     & b^{ 4}&ab^{ 6}& b^{ 9}&ab^{11}\\
a      &a      & b^{ 5}&ab^{ 5}& b^{10}&ab^{10}\\
\end{array}
\]
Consequently
$(A,B^{-1})$ is an $(26,2,5,1)$-SEDF in $D_{13}$.

The same method of construction will  produce a near factorization
of $D_{2n}$ into factors $A$ and $B$, whenever $|A| \times|B|= 2n-1$.

\subsection{An equivalent construction}

In  Huczynska {\it et al.} \cite[Example 4.1]{HucJefNep}, the following 
construction for an $(26,2,5,1)$-SEDF $(A_1,A_2)$ in $D_{13}$
is provided:
\begin{align*}
A_1 &= \{e,a,b,ab,b^2\}\\
A_2 &= \{ab^2,b^5,ab^7,b^{10},ab^{12}\}\\
\end{align*}
It is easily verified that $(A_1,A_2^{-1})$ is a near factorization of $D_{13}$:
\[
A_2^{-1}=\{ab^2,b^8,ab^7,b^{3},ab^{12}\}
\] and 
\[
\begin{array}{c|ccccc}
\cdot  &ab^{ 2}& b^{ 8}&ab^{ 7}& b^{ 3}&ab^{12}\\\hline
e      &ab^{ 2}& b^{ 8}&ab^{ 7}& b^{ 3}&ab^{12}\\
a      & b^{ 2}&ab^{ 8}& b^{ 7}&ab^{ 3}& b^{12}\\
 b     &ab     & b^{ 9}& ab^{ 6}& b^{ 4}&ab^{11}\\
ab     &b      &ab^{ 9}&b^{ 6}&ab^{ 4}& b^{11} \\
 b^2   &a      & b^{10}&ab^{ 5}& b^{ 5}&ab^{10}
\end{array}
\]
Diagramming $A_1$ and $A_2^{-1}$ we have the following:
\[\begin{array}{lc}
A_1:&\tikz[baseline=(O),scale=0.75, line width=1]
{
 \coordinate (O) at (0,1.0);
 \Bcell{0}{ 0}
 \Bcell{1}{ 0}
 \Bcell{0}{ 1}
 \Bcell{1}{ 1}
 \Bcell{0}{2}
 \foreach \i in {0,1,2} { \draw (0,\i)--(13,\i); }
 \foreach \i in {0,...,13} {\draw (\i,0)--(\i,2);}
 \node at (-0.5,2.5) {\makebox(0,0){$i =$}};
 \node at (-0.5,1.5) {\makebox(0,0){$b^i$}};
 \node at (-0.5,0.5) {\makebox(0,0){$ab^i$}};
 \foreach \i in {0,...,12} 
 {
   \node at (\i+0.5,2.5) {$\i$};
 }
}
\vspace{.25in}
\\
A_2^{-1}:&\tikz[baseline=(O),scale=0.75, line width=1]
{
 \coordinate (O) at (0,1.0);
 \Bcell{1}{ 2}
 \Bcell{0}{ 8}
 \Bcell{1}{ 7}
 \Bcell{0}{ 3}
 \Bcell{1}{12}
 \foreach \i in {0,1,2} { \draw (0,\i)--(13,\i); }
 \foreach \i in {0,...,13} {\draw (\i,0)--(\i,2);}
 \node at (-0.5,2.5) {\makebox(0,0){$i =$}};
 \node at (-0.5,1.5) {\makebox(0,0){$b^i$}};
 \node at (-0.5,0.5) {\makebox(0,0){$ab^i$}};
 \foreach \i in {0,...,12} 
 {
   \node at (\i+0.5,2.5) {$\i$};
 }
}
\end{array}
\]
\medskip

Visually, the pattern for $A_1$ and $A_2^{-1}$ seems to be the
same for $A$ and $B$. The bottom row needs  shifting and the 
diagram needs to be flipped. To see this algebraically
let $h=ab^2$. Then
\begin{align*}
A_1h 
&= \{e,a,b,ab,b^2\}h\\
&= \{e,a,b,ab,b^2\}ab^2\\
&= \{ab^2,b^2,ab,b,a\}\\ 
&= A
\intertext{and}
(A_2h)^{-1} = h^{-1}A_2^{-1}=hA_2^{-1}
&=ab^2\{ab^2,b^8,ab^7,b^{3},ab^{12}\}\\
&=\{e,ab^{10},b^5,ab^{5},b^{10}\}\\ 
&= B.
\end{align*}
Hence~\cite[Example 4.1 (ii)]{HucJefNep}
and the example in Section~\ref{example}
are equivalent.

\bigskip

\subsection{General result}
The construction generalizing the example in
Section~\ref{example} appears in~\cite{CGHK} as follows:
If $k$ is a divisor of $2n-1$, then we have the near factorization $D_n-1=AB$, where
\begin{align*}
A&= \{b^i : 1 \leq i \leq \tfrac{k-1}{2}\} \: \medmath{\bigcup} \: \{ab^{i}: 0 \leq i \leq \tfrac{k-1}{2}\}\\
\intertext{and}
B&=\{b^{ik} : 0 \leq ik < n \} \: \medmath{\bigcup} \: \{ab^{ik}:0 < ik <n\}.
\intertext{Here, we are interested in 
$(k^2 + 1, 2, k, 1)$-SEDFs in $D_{(k^2+1)/2}$, So we specialize to $2n-1 = k^2$ and obtain:}
A&=\{b^i : 1 \leq i \leq \tfrac{k-1}{2}\} \: \medmath{\bigcup} \: \{ab^{i}: 0 \leq i \leq \tfrac{k-1}{2}\}\\
\intertext{and}
B&=\{b^{ik} : 0 \leq i \leq \tfrac{k-1}{2} \} \: \medmath{\bigcup} \: \{ab^{ik}:1 \leq i \leq \tfrac{k-1}{2}\}. 
\intertext{The following construction 
of a $(k^2+1,2,k,1)$-SEDF $(A_1,A_2)$ in $D_{(k^2+1)/2}$
was published in ~\cite[Theorem 4.2]{HucJefNep}.}
A_1&=\{b^i: 0 \leq i \leq \tfrac{k-1}{2}\} \: \medmath{\bigcup} \: \{ab^i:0 \leq i \leq \tfrac{k-3}{2}\}\\
A_2&=\{b^{ik}: 1\leq i \leq \tfrac{k-1}{2}\} \: \medmath{\bigcup} \: \{ab^{ik+\tfrac{k-1}{2}}: 0\leq i \leq \tfrac{k-1}{2}\}
\end{align*}

\begin{theorem} 
In $D_{(k^2+1)/2}$, the dihedral group of order $k^2+1$,
the near factorization $(A,B)$ of $D_{(k^2+1)/2}$
and the $(k^2+1,2,k,1)$-SEDF $(A_1,A_2)$ are equivalent.
\end{theorem}
\begin{proof}
Let $h=ab^{\tfrac{k-1}{2}}$. Then $b^ih = ab^{(k-1-2i)/2}$ and $ab^ih=b^{(k-1-2i)/2}$. Hence
\begin{align*}
A_1h
&=\bigg(\left\{b^i: 0 \leq i \leq \tfrac{k-1}{2}\right\} \: \medmath{\bigcup} \: \left\{ab^i:0 \leq i \leq \tfrac{k-3}{2}\right\}\bigg)h\\
&=\bigg(\left\{b^ih: 0 \leq i \leq \tfrac{k-1}{2}\right\} \: \medmath{\bigcup} \: \left\{ab^ih:0 \leq i \leq \tfrac{k-3}{2}\right\}\bigg)h\\
&=\left\{ab^{(k-1-2i)/2}: 0 \leq i \leq \tfrac{k-1}{2}\right\} \: \medmath{\bigcup} \: \left\{b^{(k-1-2i)}:0 \leq i \leq \tfrac{k-3}{2}\right\}\\
&=\left\{ab^{i}: 0 \leq i \leq \tfrac{k-1}{2}\right\} \: \medmath{\bigcup} \: \left\{b^{i}:1 \leq i \leq \tfrac{k-1}{2}\right\} \\&= A.\\
\intertext{And}
(A_2h)^{-1}= h^{-1}A_2^{-1} = h A_2^{-1}
&=h \bigg(\left\{b^{ik}: 1\leq i \leq \tfrac{k-1}{2}\right\} \: \medmath{\bigcup} \: \left\{ab^{ik+\tfrac{k-1}{2}}: 0\leq i \leq \tfrac{k-1}{2}\right\}\bigg)^{-1}\\
&=h \bigg(\left\{b^{-ik}: 1\leq i \leq \tfrac{k-1}{2}\right\} \: \medmath{\bigcup} \: \left\{ab^{ik+\tfrac{k-1}{2}}: 0\leq i \leq \tfrac{k-1}{2}\right\}\bigg)\\
&=\left\{hb^{-ik}: 1\leq i \leq \tfrac{k-1}{2}\right\} \: \medmath{\bigcup} \: \left\{hab^{ik+\tfrac{k-1}{2}}: 0\leq i \leq \tfrac{k-1}{2}\right\}\\
&=\left\{ab^{\tfrac{k-1}{2}-ik}: 1\leq i \leq \tfrac{k-1}{2}\right\} \: \medmath{\bigcup} \: \left\{b^{ik}: 0\leq i \leq \tfrac{k-1}{2}\right\}.\\
\intertext{We now use that  the identity $e=b^n$ and the assumption that $2n-1 = k^2$ to see that $b^{-1}= b^{k^2}$. Thus}
h^{-1}A_2^{-1} 
&=\left\{ab^{\tfrac{k+k^2}{2}-ik}: 1\leq i \leq \tfrac{k-1}{2}\right\} \: \medmath{\bigcup} \: \left\{b^{ik}: 0\leq i \leq \tfrac{k-1}{2}\right\}\\
&=\left\{ab^{\left( \tfrac{k+1}{2}-i \right) k}: 1\leq i \leq \tfrac{k-1}{2}\right\} \: \medmath{\bigcup} \: \left\{b^{ik}: 0\leq i \leq \tfrac{k-1}{2}\right\}\\
&=\left\{ab^{ik}: 1\leq i \leq \tfrac{k-1}{2}\right\} \: \medmath{\bigcup} \: \left\{b^{ik}: 0\leq i \leq \tfrac{k-1}{2}\right\} \\&= B.\\
\end{align*}
\end{proof}
\noindent
This near factorization and the equivalent SEDFs are the only known
examples in non-abelian groups.

\section{Summary}

There remain some interesting open problems. We mention some of them now.

\begin{enumerate}
\item Does there exist an $(a^2+1, 2, a, 1)$-SEDF in  $\zed_{a^2+1}$ that is not affine-equivalent to an $\alpha$-valuation? 
\item Is there a simple way to characterize when two sequences of blow-up operations yield affine-equivalent SEDFs?
\item Is there a generalization of $\alpha$-valuations that would provide constructions of $(a^2+1, 2, a, 1)$-SEDFs in non-cyclic groups? 
\item Do there exist nonequivalent SEDFs in dihedral groups?
\item Does there exist an SEDF in a non-abelian group that is not a dihedral group? 
\end{enumerate}


\begin{thebibliography}{10}

\bibitem{BJWZ}
J. Bao, L. Ji, R. Wei and Y. Zhang.
New existence and nonexistence results for strong external difference families
\emph{Discrete Mathematics} {\bf 341} (2018), 1798--1805.

\bibitem{CGHK} D. de Caen, D.A. Gregory, I.G. Hughes and D.L. Kreher.
Near-factors of finite groups.
\emph{Ars Combin.} {\bf 29} (1990), 53--63.


\bibitem{HuPa}
S. Huczynska and M. Paterson.
Existence and non-existence results for strong external difference families.
\emph{Discrete Mathematics} {\bf 341} (2018), 87--95.

\bibitem{HucJefNep}
S. Huczynska, C. Jefferson and S.  Nep\v{s}insk\'{a}. Strong external difference
  families in abelian and non-abelian groups.
\emph{Cryptography Commun.} \textbf{13} (2021), 331--341.

\bibitem{JeLi}
J. Jedwab and S. Li.
Construction and nonexistence of strong external difference families. 
\emph{J. Algebr. Comb.} {\bf 49} (2019), 21--48.

\bibitem{libexact}
P.\ Kaski and O.\ Pottonen.
Libexact.  
\url{http://pottonen.kapsi.fi/libexact.html}.

\bibitem{Knuth}
D.E. Knuth,
\textsl{Dancing Links}. In
``Millennial Perspectives in Computer Science'',
J. Davies, B. Roscoe, and J. Woodcock, Eds.,
Palgrave, Basingstoke, England, 2000,
pp. 187--214.

\bibitem{KM76}
E. Kramer and D. Mesner. $t$-designs on hypergraphs.
\emph{Discr. Math.} {\bf 15} (1976), 263--296.

\bibitem{LeLiPr} 
K.H. Leung, S. Li and T.F. Prabowo.
Nonexistence of strong external difference families in abelian groups of order being product of at most three primes.
\emph{Journal of Combinatorial Theory A} \textbf{178} (2021),   105338.



\bibitem{LePr} 
K.H. Leung and T.F. Prabowo.
Some nonexistence results for $(v,m,k,pq)$-strong external difference families.
\emph{Journal of Combinatorial Theory A} \textbf{187} (2022),   105575.



\bibitem{MaSt} 
W.J.~Martin and D.R. Stinson.
Some nonexistence results for strong external difference families using character theory.
\emph{Bull. Inst. Combin. Appl.} \textbf{80} (2017),   79--92.

\bibitem{PS16}
M.B. Paterson and D.R. Stinson.
Combinatorial characterizations of algebraic manipulation
detection codes involving generalized difference families.
\emph{Discrete Math.} \textbf{339} (2016),  2891--2906.

\bibitem{PS24}
M.B. Paterson and D.R. Stinson.
Circular external difference families, graceful labellings and cyclotomy.
\emph{Discrete Math.}  \textbf{347} (2024), article 114103, 15 pp.

\bibitem{rosa}
A. Rosa. On certain valuations of the vertices of a graph.
\newblock In: P.~Rosenstiehl (ed.) Theory of Graphs (International Symposium,
  Rome, 1966), pp. 349--355 (1967).

\bibitem{WYFF}
J. Wen. M. Yang, F. Fu and K. Feng.
Cyclotomic construction of strong external difference families in finite fields.
\emph{Designs, Codes and Cryptography}  {\bf 86} (2018), 1149--1159.


\end{thebibliography}
\end{document}